\documentclass[12pt,reqno]{amsart}

\usepackage{amsmath}
\usepackage{amsthm}
\usepackage{amssymb}
\usepackage{mathrsfs}
\usepackage{enumitem}
\usepackage{tikz}

\numberwithin{equation}{section}
\newtheorem{theo}{Theorem}
\numberwithin{theo}{section}
\newtheorem{prop}[theo]{Proposition}

\newtheorem{coro}[theo]{Corollary}
\newtheorem{lemma}[theo]{Lemma}
\newtheorem{claim}[theo]{Claim}

\theoremstyle{definition}
\newtheorem{defi}[theo]{Definition}

\numberwithin{notation}{section}

\newcommand{\F}{\mathcal{F}}
\newcommand{\s}{\subseteq}

\newcommand{\w}{\omega}
\newcommand{\al}{\alpha}

\newcommand{\be}{\beta}

\newcommand{\calF}{\mathcal F}

\newcommand{\PP}{\mathbb P}

\newcommand{\CC}{\mathbb{C}}
\newcommand{\KK}{\mathbb K}

\newcommand{\bff}{\mathbf F}

\DeclareMathOperator{\supp}{supp}

\title{Forcing and Construction Schemes}
\author{Damjan Kalajdzievski}
\author{Fulgencio Lopez}
\address{Department of Mathematics and Statistics\\ York University \\ N520 Ross
4700 Keele Street\\ Toronto, ON M3J 1P3, Canada.}
\email{dkala011@mathstat.yorku.ca}
\address{Department of Mathematics\\ University of Toronto\\ Bahen Center 40 St. George St.\\ Toronto, Ontario M5S 2E4, Canada.}
\email{fulgencio.lopez@mail.utoronto.ca}


\begin{document}

\begin{abstract}
We investigate forcing and independence questions relating to construction schemes.
We show that adding $\kappa\geq\omega_1$ Cohen reals adds a capturing construction scheme.
We study the weaker structure of $n$-capturing construction schemes and 
show that it is consistent to have $n$-capturing construction schemes but no $(n+1)$-capturing 
construction schemes.
We also study the relation of $n$-capturing with the $m$-Knaster hierarchy
and show that  MA$_{\omega_1}($K$_m)$ and $n$-capturing are independent if $n\leq m$ and incompatible if $n>m$.
\end{abstract}

\maketitle

\section{Introduction and Preliminaries}
Following some forcing construction of \cite{BGT} and \cite{LAT}, Todor\v{c}evi\'{c}~\cite{todor} 
introduced the concept of capturing construction schemes and showed they exist
if you assume $\diamondsuit$. Also in \cite{todor}, capturing construction schemes are used to construct a Banach space of the form $C(K)$ without biorthogonal sequences and other objects from \cite{LAT}, but without forcing.
These results are characterized for the recursive nature of the proofs which makes building counterexamples more intuitive.
In \cite{ful} it was proved that some of the more complex examples of Banach spaces from \cite{LAT} also follow from the existence of 
a capturing construction scheme.
In \cite{LT} it was proved that Suslin trees and T-gaps exists if we assume there are capturing construction schemes.
These latest results only require a weaker version of capturing, namely 3-capturing.

In this work we study the consistency of capturing construction schemes and the weaker versions of capturing.
Throughout this work, when we talk about the consistency of \emph{capturing} we mean 
\textit{for every type $(m_k,n_k,r_k)_k$ there is a capturing construction scheme $\calF$ of that type} (see below for definitions),
and analogously when we talk about the consistency of $n$-\emph{capturing}.

We start by studying the consistency of $n$-capturing and $m$-\emph{Knaster} (for $2\leq n,m$) in Section~\ref{ncap.nkna}.
We show the consistency of $n$-capturing with no $(n+1)$-capturing construction scheme,
and relate this with the $m$-Knaster Hierarchy.
Recall that a forcing notion $\PP$ is said to be $m$-\emph{Knaster} (K$_m$) (for $2\leq m$ ) if for every uncountable $W\subset\PP$ there is 
$W_0\subset W$ uncountable such that for every $p_1,\ldots,p_m\in W_0$ there is $p\in\PP$ with $p\leq p_1,\ldots, p_m$.
MA$_{\omega_1}($K$_m)$ is the forcing axiom for $\aleph_1$ dense sets on K$_m$ forcings.
The main result of this Section is:
\begin{theo}\label{ncap}
$n$-capturing is independent of MA$_{\omega_1}(K_m)$ if $n\leq m$, and they are incompatible if $n>m$.
Also capturing is independent of MA$_{\omega_1}($precaliber $\aleph_1)$.
\end{theo}
The proof of this Theorem gives an alternative argument to the well known fact that
MA$_{\omega_1}($K$_m)\not\Leftarrow$MA$_{\omega_1}($K$_{m+1})$.

Section~\ref{cohen} is dedicated to proving the consistency of capturing using forcing. 
\begin{theo}\label{ch5.cohen}
Adding $\kappa\geq\aleph_1$ Cohen reals also adds a capturing construction scheme. 
\end{theo}

In Section~\ref{extra} we show other versions of capturing are also consistent. They are forced by adding $\aleph_1$ Cohen reals.

We follow standard notation in Set Theory (see for example Kunen~\cite{Kunen}).
When we refer to a $\Delta$-System $(s_\alpha:\alpha<\omega_1)$ we mean $s_\alpha$ are finite subsets of $\omega_1$,
and for every $\alpha<\beta<\omega_1$ we have $s_\alpha\cap s_\beta=s$ for some $s\subset\omega_1$ fixed,
 and $\max(s_\alpha)<\min(s_\beta)$.

Let us introduce the main concept of this work.
\begin{defi}
Let $(m_k)_{k<\omega}, (n_k)_{1\leq k<\omega}$ and $(r_k)_{1\leq k<\omega}$ be sequences of natural numbers such that
$m_0=1$, $m_{k-1}>r_k$ for all $k>0$, $n_k>k$ and for every $r<\omega$ there are infinitely many $k$'s with $r_k=r$. If for every $k>0$ we have
$$ m_k=n_k(m_{k-1}-r_k)+r_k$$
we say that $(m_k,n_k,r_k)_{k<\omega}$ forms a {\it type}.
\end{defi}

\begin{defi}\label{cons.sch}
We say that $\calF$ is a \emph{construction scheme of type} $(m_k,n_k,r_k)_{k<\omega}$, if $\calF\subset[\omega_1]^{<\omega}$ is a family of finite subsets of $\omega_1$, partitioned into levels $\calF=\bigcup_{k<\omega}\calF_k$, such that for every $F\in\calF$, there is $R(F)\sqsubset F$ with:
\begin{enumerate}
\item For every $A\subset\omega_1$ finite, there is $F\in\calF$ such that $A\subset F$.
\item $\forall F\in\calF_k$, $|F|=m_k$ and $|R(F)|=r_k$.
\item For all $F,E\in\calF_k$, $E\cap F\sqsubset F,E$.
\item $\forall F\in\calF_k$, there are unique $F_0,\ldots, F_{n-1}\in\calF_{k-1}$ with
$$F=\bigcup_{i<n} F_i $$
 Furthermore $n=n_k$ and $(F_i)_{i<n_k}$ forms an increasing $\Delta$-system with root $R(F)$, i.e.,
$$R(F)< F_0\setminus R(F) <\ldots < F_{n_k-1}\setminus R(F) $$
\end{enumerate}

We call the sequence $(F_i)_{i<n_k}$ of (4) the \emph{canonical decomposition} of $F$.
\end{defi}

For every type $(m_k,n_k,r_k)_{k<\omega}$, there is a construction scheme $\calF$ of that type in ZFC.
This is implicit in the proof of the consistency with $\diamondsuit$ of \cite{todor}.
It is also easy to show that for any type $(m_k,n_k,r_k)_{k<\omega}$ there is a construction scheme $\calF^\omega$ in $\omega$
(i.e, $\calF^\omega$ is a family of finite subsets of $\omega$ and it is cofinal in $\omega$ the rest is as the definition above).
We will use this fact in Section~\ref{cohen}.

To avoid confusion we will use $m_k, n_k$, and $r_k$ as above, and we will omit reference to the type of a construction scheme.
For $F\in\calF$ and $F=\bigcup_{i<n_k}F_i$ the canonical decomposition of $F$.
We write $\varphi_i:F_0\rightarrow F_i$ for the unique order preserving bijection between $F_0$ and $F_i$, or for $E,F\in\F_k$ write  $\varphi_{F,E}:F\rightarrow E$ for the unique order preserving bijection between $F$ and $E$.
Analogously, if $f$ is a function on $F_0$ we can define the function $\varphi_i(f)=f\circ\varphi_i^{-1}$ in $F_i$.

The following lemma tells us more about the structure of a construction scheme

\begin{lemma}\label{construction.properties}
For $F\in\calF_k$, $E\in\calF_l$, with $l\leq k$ we have $E\cap F\sqsubseteq E$.
\end{lemma}
\begin{proof}
We prove the lemma by induction on $k$ and $l$.
If $l=k$ the result follows by the properties of $\calF$.
It's enough to show that: if it holds for $l\leq k-1$, it holds for $l$ and $k$ as well.
Let $F$ as above and let $F=\bigcup_{i<n_{k}}F_i$ be its canonical decomposition. 
Since the $F_i$'s are in $\calF_{k-1}$ we can apply our hypothesis and $E\cap F_i\sqsubseteq E$ for 
every $i<n_k$. If $E\cap(F\setminus R(F))=\emptyset$ then the result follows, otherwise
let $i<n_k$ be minimal such that $E\cap(F_i\setminus R(F))\neq\emptyset$ then $E\cap F=E\cap F_i$.
Because if not, there is $i<j<n_{k+1}$ with $E\cap F_j\not\sqsubseteq E$.
Thus we have $E\cap F=E\cap F_i\sqsubset E$ and the result follows.
\end{proof}

\begin{coro}\label{construction.properties3}
For $F\in\calF_k$, $E\in\calF_l$ and $F=\bigcup_{i<n_k}F_i$ the canonical decomposition of $F$.
If $E\subset F$ and $l<k$ then there is some $i<n_k$ with $E\subset F_i$.
In particular, if $l=k-1$ we have $E=F_i$.
\end{coro}

\begin{coro}\label{construction.properties4}
Let $E,F\in\calF_k$, then $\varphi_{E,F}(\calF\restriction E)=\calF\restriction F$. Where $\calF\restriction F=\{L\in\calF: L\subset F\}$.
\end{coro}

\begin{lemma}\label{construction.properties2}
For $F\in\calF_k$, $E\in\calF_l$ and $E\subset F$ (in particular $l\leq k$). For every $\mu\in E$ there is a copy $E^*$ of $E$ in $F$ such that
\begin{enumerate}
\item $E^*\cap (\mu+1)=E\cap(\mu+1)$.
\item $E^*\setminus\mu$ is an interval of $F$ with $\mu\in E$.
\end{enumerate}
\end{lemma}

\begin{proof}
We prove the lemma by induction on $k$ and $l$.
The result follows for $l=k$. Suppose $l<k$ is given and the result holds for $l\leq k-1$. 
Take $F=\bigcup_{i<n_k}F_i$, the canonical decomposition of $F$.
By Corollary~\ref{construction.properties3} there is $i<n_k$ such that $E\subset F_i$. By the induction hypothesis there is $E^{**}$ 
a copy of $E$ in $F_i$ such that the conclusion holds. If $\mu\notin R(F)$ then $E^*=E^{**}$ works.
Otherwise, let $E^*=\varphi_{F_i,F_0}(E^{**})$ by Corollary~\ref{construction.properties4}, $E^*$ is a copy of $E$ and $E^*\setminus\mu$ is an interval of $F_0$. Since $\mu\in R(E)$ then (1) holds, and (2) holds because $F_0$ is an interval of $F$.
\end{proof}

The following concept is useful to construct many other results as mentioned above.

\begin{defi}
Let $\calF$ be a construction scheme, and $2\leq n$. We say that $\calF$ is $n$-\emph{capturing} if for every uncountable $\Delta$-system
$(s_\xi)_{\xi<\omega_1}$ of finite subsets of $\omega_1$ with root $s$, there are $\xi_0<\ldots<\xi_{n-1}<\omega_1$,
and $F\in\calF$ with canonical decomposition $F=\bigcup_{i<n_k}F_i$, such that
\begin{align*}
                                                      &s                  \subset R(F)\\
   \mbox{for every $i<n$,} \quad  &                    s_{\xi_i}\setminus s  \subset F_i\setminus R(F), \\
 \mbox{for every $i<n$,} \quad     &      \varphi_i(s_{\xi_0})=s_{\xi_i}.
\end{align*}

We say that $\calF$ is \emph{capturing} if $\calF$ is $n$-capturing for every $n<\omega$.
\end{defi}

\section{The Hierarchies of $n$-Knaster and $n$-capturing}\label{ncap.nkna}
Recall that MA$_{\omega_1}(K_n)$ implies MA$_{\omega_1}(K_m)$ for every $m\geq n$,
whereas $n$-capturing implies $m$-capturing for every $m\leq n$.
Thus, we have the following two hierarchies:

\begin{center}
\includegraphics[width=15cm]{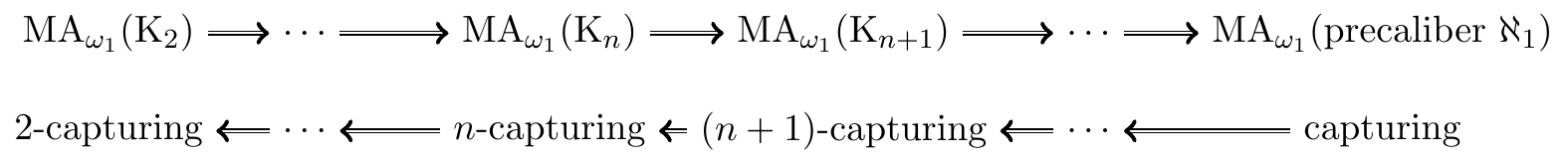}
\end{center}

The main result of this section give us a relation between this two types of axioms and shows that none of the implications
above can be reversed.
\renewcommand{\thetheo}{\ref{ncap}}
\begin{theo}
$n$-capturing is independent of MA$_{\omega_1}(K_m)$ if $n\leq m$, and they are incompatible if $n>m$.
Also capturing is independent of MA$_{\omega_1}($precaliber $\aleph_1)$.
\end{theo}
\renewcommand{\thetheo}{\thesection.\arabic{theo}}
\addtocounter{theo}{-1}
We start the analysis of $n$-capturing with the following Preservation Lemma.

\begin{lemma}\label{lem1}
Capturing is preserved by $K_n$ forcing notions.
Let $\PP$ be a $K_n$ forcing notion and let $\calF$ be a $n$-capturing construction scheme on $V$.
If $G\subset\PP$ is a generic filter for $\PP$, then $\check{\calF}$ is a $n$-capturing construction scheme on $V[G]$.
In particular capturing is preserved by precaliber $\aleph_1$ forcing notions.
\end{lemma}
\begin{proof}
Let $\PP$ be a $K_n$ forcing notion and $\dot{\Gamma}$ a $\PP$-name for an uncountable subset of $\omega_1$.
Let $W\subset\omega_1$ and $p_\alpha\in\PP$, $\alpha\in W$ such that
$$ p_\alpha\Vdash \alpha\in \dot{\Gamma}$$
for every $\alpha\in W$.
Since $\PP$ is $K_n$ there is $n$-linked $W_0\subset W$ uncountable.
Recall $\calF$ is $n$-capturing in $V$, therefore there are $\alpha_0<\ldots<\alpha_{n-1}$ in $W_0$ which are captured by $\calF$.
We find now $q\in\PP$ with $q\leq p_0,\ldots, p_{n-1}$, then
$$q\Vdash \alpha_0,\ldots,\alpha_{n-1}\in\dot{\Gamma},\mbox{ and they are captured by $\check{\calF}$.}$$
\end{proof}

Consider the following property
\begin{itemize}
\item[$(\bigstar)_m$ ] For every $\Gamma\subset\omega^\omega$ there is $\Gamma_0\subset\Gamma$ uncountable 
such that $\Gamma_0$ has no $g_0,\ldots,g_m$ and $k<\omega$ with
$g_0\restriction k=\ldots=g_m\restriction k$, and $|\{g_0(k),\ldots,g_m(k)\}|=m+1$.
\end{itemize}
Recall the following result of Todor\v{c}evi\'{c} implicit in \cite{Tcomb}

\begin{theo}[Todor\v{c}evi\'{c}]
MA$_{\omega_1}(K_m)$ implies $(\bigstar)_m$.
\end{theo}

The following result proves the first half of Theorem~\ref{ncap}

\begin{theo}\label{theorem3}
Let $\calF$ be a $(m+1)$-capturing construction scheme. Then $(\bigstar)_m$ fails.
\end{theo}
\begin{proof}
Let $\calF$ be as above. For every $F\in\calF_l$ we construct, inductively on $l$, $(f_\alpha^F:(l+1)\rightarrow N_l)_{\alpha<\omega_1}$
such that
\begin{enumerate}
\item for $E,F\in\calF_l$ and $\varphi:E\rightarrow F$ the increasing bijection between $E$ and $F$, for every $\alpha\in E$,
        if $\beta=\varphi(\alpha)$ then $f^F_\beta=f^E_{\alpha}$.
\item for $E\in\calF_{l_0}$ and $F\in\calF_{l_1}$, $l_0<l_1$, if $\alpha\in E\cap F$ then $f^F_\alpha\restriction(l_0+1)=f_\alpha^E$.
\end{enumerate}

Let $F\in\calF_k$ with canonical decomposition $F=\bigcup_{i<n_k}F_i$ and suppose $(f_\alpha^{F_i}:\alpha\in F_i)$ is defined for all $i<n_k$ satisfying (1) and (2) above. Let $f_\alpha^F=\emptyset$ if $\alpha\notin F$.

For $\alpha\in R(F)$ let $f^F_\alpha(k)=N_{k-1}$ and $f_\alpha^F\restriction k=f_\alpha^{F_0}$.

For $\alpha_0\in F_0\setminus R(F)$ and $\alpha_i=\varphi_i(\alpha)$, $i<n_k$. We let $f^F_{\alpha_i}\restriction k=f^{F_i}_{\alpha_i}$ and
$$ f^F_{\alpha_i}=N_{k-1}+i+1$$
And let $N_k=N_{k-1}+n_k+1$.

It is easy to see that (1) and (2) hold, and so $f_\alpha=\bigcup_{F\in\calF}f_\alpha^F$ is a well defined function.
Then $\Gamma=\{f_\alpha:\alpha<\omega_1\}$ is a witness to the failure of $(\bigstar)_m$.
To see this suppose $\Gamma_0=\{f_\alpha:\alpha\in W\}$ where $W\subset\omega_1$ is uncountable.
Since $\calF$ is $(m+1)$-capturing there are $\xi_0<\ldots<\xi_m$ in $W$ captured by some $F\in\calF_k$.
This implies $f_{\xi_0}\restriction k=\ldots=f_{\xi_m}\restriction k$ and $|\{f_{\xi_0}(k),\ldots,f_{\xi_m}(k)\}|=m+1$, and hence $(\bigstar)_m$ fails as we wanted to show.
\end{proof}

\begin{proof}[Proof of Theorem~\ref{ncap}]
Start by assuming $n\leq m$.
To see $n$-capturing is independent of MA$_{\omega_1}$(K$_m)$, note that any model of MA$_{\omega_1}$
is also a model of MA$_{\omega_1}$(K$_m)$ and contains no $n$-capturing construction scheme for any $2\leq n<\omega$ (see \cite{LT} for $n>2$, and see Proposition~\ref{2cap} of this paper for $n=2$).
Thus, it is consistent to have MA$_{\omega_1}$(K$_m)$ and no $n$-capturing construction schemes.
To show the other direction, start with a model $V$ that has a capturing construction scheme $\calF$.
Let $\KK_m$ be the K$_m$ poset that forces MA$_{\omega_1}$(K$_m)$.
Then $\calF$ remains $m$-capturing on the extension by Lemma~\ref{lem1} hence it is $n$-capturing provided $n\leq m$.

Suppose now $n>m$ and $V$ is a model of MA$_{\omega_1}$(K$_m)$, then $(\bigstar)_m$ holds on $V$.
By Theorem~\ref{theorem3} we know $V$ contains no $(m+1)$-capturing construction scheme,
otherwise $(\bigstar)_m$ fails which is a contradiction.
Thus $V$ has no $n$-capturing construction scheme for $n>m$, as we wanted to show.

To see MA$_{\omega_1}$(precaliber $\aleph_1)$ and capturing are independent we proceed in the same manner.
Any model of MA$_{\omega_1}$ satisfies MA$_{\omega_1}$(precaliber $\aleph_1)$ and has no capturing construction scheme.
Finally, let $V$ be a model that contains a capturing construction scheme.
Let $\KK$ be a forcing notion with precaliber $\aleph_1$ that forces MA$_{\omega_1}$(precaliber $\aleph_1)$.
Since $\KK$ has precaliber $\aleph_1$, $\calF$ remains capturing in the extension.
This finishes the proof.
\end{proof}

It is interesting to find a K$_n$ forcing notion that kills $(n+1)$-capturing in an obvious way.
Suppose $\calF$ is a capturing construction scheme. Let $\calF$ be fixed. 
\begin{defi}
Let $P\in\PP_n$ if $\calF$ does not capture $\bigl\{\{\xi_i\}:i\leq n\bigr\}$ for any $\xi_0<\ldots<\xi_n$ in $P$.
We say $P\leq Q$ if $Q\subset P$.
\end{defi}

\begin{lemma}
$\PP_n$ defined as above is $K_n$.
\end{lemma}
\begin{proof}
Take $(P_\alpha:\alpha<\omega_1)\subset\PP_n$. 
We can find $D_\alpha\in\calF_{k_\alpha}$ such that $P_\alpha\subset D_\alpha$.

Find $\Gamma\subset\omega_1$ uncountable, and $k<\omega$ such that
\begin{enumerate}
\item $(D_\alpha:\alpha\in\Gamma)$ forms a $\Delta$-System,
\item $k_\alpha=k$ for all $\alpha\in \Gamma$, and
\item for every $\alpha<\beta$ in $\Gamma$, we have $\varphi_{D_\alpha,D_\beta}(P_\alpha)=P_\beta$.
\end{enumerate}

Note that (2) and (3) imply that for all $\alpha<\beta$, $\xi\in D_\alpha\cap D_\beta$, 
then $\xi\in P_\alpha$ if and only if $\xi\in P_\beta$.

We show $(P_\alpha:\alpha\in\Gamma)$ is $n$-linked.
Take $\alpha_0<\ldots<\alpha_{n-1}$ in $\Gamma$.
Let $Q=\bigcup_{i<n}P_i$.
Suppose $\xi_0<\ldots<\xi_n$ are in $Q$ and $F\in\calF_\ell$ captures $\bigl\{\{\xi_i\}:i\leq n\bigr\}$.
Take $F=\bigcup_{i<n_\ell}F_i$ the canonical decomposition of $F$.
We must have 
\begin{equation}\label{eq1}
\begin{aligned}
 \xi_i&\in F_i\setminus R(F)\\
\bigl(F_i\setminus R(F):i<n_l\bigr)&\mbox{ are pairwise disjoint}
\end{aligned}
\end{equation}

Let us get a contradiction. 

\textbf{Case $l\leq k$:}
Let $j\leq n$ with $\xi_n\in P_{\alpha_j}$.
Applying Proposition~\ref{construction.properties}, $F\cap D_{\alpha_j}\sqsubseteq F$.
Therefore $\xi_0,\ldots,\xi_n\in D_{\alpha_j}$
which implies $\xi_0,\ldots,\xi_n\in P_{\alpha_j}$.
But $\calF$ captures $\bigl\{\{\xi_i\}:i\leq n\bigr\}$ and this is a contradiction because $P_{\alpha_j}\in\PP_n$.

\textbf{Case $l>k$:}
There is some $j<n$ and $i_0<i_1\leq n$ such that $\xi_{i_0},\xi_{i_1}\in P_{\alpha_j}$.
Then $F_{i_1}\in\calF_{\ell-1}$, and $F_{i_1}\cap D_{\alpha_j}\sqsubseteq D_{\alpha_j}$ by Proposition~\ref{construction.properties},
but this implies $\xi_{i_0}\in F_{i_1}$.
This contradicts \eqref{eq1} 

We conclude that for every $\xi_0<\ldots\xi_n$ in $Q$, $\calF$ does not capture $\bigl\{\{\xi_i\}:i\leq n\bigr\}$.
Hence $Q\in\PP_n$. It is clear that $Q\leq P_{\alpha_i}$ for $i<n$.
This finishes the proof.
\end{proof}

It is clear that $\PP_n$ kills $(n+1)$-capturing, thus we have an explicit proof that 
MA$_{\aleph_1}($K$_n$) is incompatible with $m$-capturing for $m>n$.

Assume $m>2$ and note that the model obtained in the proof of Theorem~\ref{ncap}, which starts with a capturing construction scheme and then forces $\mbox{MA}_{\omega_1}(\mbox{K}_m)$, shows
the consistency of 
$$ \mbox{MA}_{\omega_1}(\mbox{K}_m) + m\mbox{-capturing } +\neg (m+1)\mbox{-capturing }+\neg\mbox{MA}_{\omega_1}(\mbox{K}_{m-1})$$
this gives us an alternative proof of MA$_{\omega_1}($K$_m)\not\Leftarrow $MA$_{\omega_1}($K$_{m+1}$)
showing that the hierarchy of $m$-Knaster forcing axioms is strict. 

To get that MA$_{\w_1}$ implies there are no $2$-capturing construction schemes, we prove the following:
\begin{prop}\label{2cap}
If $\F$ is $2$-capturing, then $\PP_1$ is c.c.c.
\end{prop}
\begin{proof}
Suppose $(P_\alpha:\alpha<\omega_1)\subset\PP_1$ forms an uncountable antichain, and refine this family so that it forms a $\Delta$-system. Since $\F$ is $2$-capturing, we can recursively construct a family $(D_\alpha:\alpha\in \Gamma)\s\F$ and refine it so that $(D_\alpha:\alpha\in\Gamma)\s\F_k$ forms an uncountable $\Delta$-System, and for $\al\in\Gamma$, $D_\alpha$ captures some $(P_{\al'},P_{\al''})$. Again, since $\F$ is $2$-capturing, there are some $F\in\F$, $\al<\be\in\Gamma$, such that $F$ captures $(D_{\al},D_{\be})$. 

We claim that $P_{\al'}\cup P_{\be''}\in \PP_1$, which finishes the proof with a contradiction. Suppose $\xi_0<\xi_1\in P_{\al'}\cup P_{\be''}$ are captured by some $E\in\F_l$. Note that since $P_{\al'}, P_{\be''}\in \PP_1$, $\xi_0\in P_{\al'}\setminus P_{\be''},\ \xi_1\in P_{\be''}\setminus P_{\al'}$, and so $\xi_0\in D_{\al}\setminus D_{\be},\ \xi_1\in  D_{\be}\setminus D_{\al}$. Let $E=\bigcup_{i<n_l}E_i$, $D_{\be}=\bigcup_{i<n_k}(D_{\be})_i$, $D_{\al}=\bigcup_{i<n_k}(D_{\al})_i$ be the respective canonical decompositions.

\textbf{Case $l\leq k$:}
Applying Proposition~\ref{construction.properties}, $E\cap D_{\be}\sqsubseteq E$.
Therefore $\xi_1\in E\cap D_{\be}$ gives $\xi_0\in D_{\be}$, and this is a contradiction.

\textbf{Case $l>k$:}
Recall that $\phi_{E_0,E_1}(\F\restriction E_0)=\F\restriction E_1$, and $D_\al$ capturing $(P_{\al'},P_{\al''})$ implies $\xi_0\in (D_\al)_0\setminus R(D_\al)$. Since there is some $E'\s E$ with $\xi_0\in E'\in \F_{k}$, and $\xi_0\in (D_\al)_0\in \F_{k-1}$, we get that $\xi_0$ must be in the $0$'th component of the canonical decomposition of $E'$, and hence $\phi_{E_0,E_1}(\xi_0)=\xi_1$ must be in the $0$'th component of the canonical decomposition of some element in $\F_k\restriction E_1$, which contradicts $\xi_1\in(D_\be)_1\setminus R(D_\be)$.

\end{proof}

\section{Capturing Construction Schemes in the Cohen Model}\label{cohen}
We dedicate this section to the proof of the following result.

\renewcommand{\thetheo}{\ref{ch5.cohen}}
\begin{theo}
Adding $\kappa\geq\aleph_1$ Cohen reals also adds a capturing construction scheme. 
\end{theo}
\renewcommand{\thetheo}{\thesection.\arabic{theo}}
\addtocounter{theo}{-1}

\begin{proof}
Assume first that $\kappa=\aleph_1$. 
We start by fixing $\calF^\omega$, a construction scheme on $\omega$ with the following property:
\begin{equation}\label{omegaCS}
\mbox{\begin{minipage}{.9\textwidth}
For every $A\subset\omega$ finite and $a<\omega$ there is $F\in\calF$ with canonical decomposition
$\bigcup_{i<n_k}F_i$, such that $A\subset F_0$ and $R(F)=F_0\cap a$.
\end{minipage}}
\end{equation}

\begin{defi}
Let $p\in\PP$ if and only if $\supp({p})\subset\omega_1$ finite, for every 
$\delta\in\supp(p)$, $\delta$ is limit, $p(\delta)=(D^p_\delta,a^p_\delta)$ where
$D^p_\delta\in\calF^\omega$, $a^p_\delta\in D^p_\delta$, and for every $\delta_0<\delta_1$ in $\supp(p)$
\begin{enumerate}
\item $D^p_{\delta_0}\subseteq D^p_{\delta_1}$, and 
\item $a^p_{\delta_0}<a^p_{\delta_1}$
\end{enumerate}

We say ${p}\leq{q}$ if $\supp({q})\subset\supp({p})$, 
\begin{enumerate}[label=(\roman*)]
\item for every $\delta<\delta'\in\supp({q})$, $a^p_{\delta'}-a^p_{\delta}\geq a^q_{\delta'}-a^q_{\delta}$, and 
\item for every $\delta\in\supp({q})$ with $D^q_\delta\in\calF_k$,
there is $W\in\calF_k$ with $W\cap a^p_\delta$ having the same size that $D_\delta^q\cap a^q_\delta$, and
 $W\setminus a^p_\delta$ is an interval of $D^p_\delta$ with $a^p_\delta\in W$.
\end{enumerate}

We say ${p}\sim{q}$ if ${p}\leq{q}$ and ${q}\leq{p}$.
\end{defi}

Note that $\PP$ is equivalent to the forcing $\CC_{\omega_1}$ for adding $\omega_1$ Cohen reals. To see this, notice that $\PP$ is a dense suborder of the partial order which is defined as above minus conditions (1), (2), or (i), and this partial order is finite support product of countable partial orders.

Now for $a<\delta<\omega_1$, define the function $\phi_{a,\delta}:\omega\rightarrow\omega_1$ by
$$\phi_{a,\delta}(\alpha)=\begin{cases} \alpha \quad   &\alpha< a \\
                                        \delta+i \quad &\alpha=a+i
											    \end{cases}$$
Note that, for $\varphi:\omega\rightarrow\omega$ increasing we have 
\begin{center}
\begin{tikzpicture}
\node (a) at (-2,2) {$\omega$};
\node (x) at (2,2) {$\omega_1$};
\node (b) at (-2,0) {$\omega$};
\draw[->] (a)--node[left] {$\varphi$}(b);
\draw[->] (a)--node[above] {$\phi_{a,\delta}$}(x);
\draw[->] (b)to node[below] {$ \quad\phi_{\varphi(a),\delta}$}(x);
\end{tikzpicture}
\end{center}
\begin{equation}\label{cohen.cap.eq1}
\phi_{\varphi(a),\delta}\circ\varphi=\phi_{a,\delta}
\end{equation}

Now let $p\in\PP$, with $\supp(p)=(\delta_0<\ldots<\delta_n)$ suppose $p(\delta_i)=(D_i,a_i)$. 
We define $\Phi^p:D_n\rightarrow\omega_1$ as:
$$
\Phi^q(x)=\begin{cases} x & x<a_0 \\
           \phi_{a_i,\delta_i}(x)& a_i\leq x<a_{i+1} \\
					 \phi_{a_n,\delta_n}(x) & x\geq a_n
          \end{cases}
$$
Finally, for $p$ as above we define $$\bff_p=\Phi^p\Bigl(\calF^\omega\restriction D_{n}\Bigr).$$ It is left as an exercise to check that $q\leq p\in \PP$ implies $\bff_p\s\bff_q$.

Let $G\subset\PP$ be a generic filter, we define $\calF$ in $V[G]$ as
$$\calF=\bigcup_{p\in G}\bff_p$$

\begin{claim}
Given $ p\in\PP$ and $\xi<\omega_1$ there is $q\leq p$ and $x<\omega$ such that $\Phi^q(x)=\xi$
\end{claim}
\begin{proof}

Let $\xi<\omega_1$ and $ p\in\PP$.
We want to find $q\leq p$ and $x<\omega$ such that $\xi=\Phi^q(x)$.

Take $\delta<\omega_1$ limit such that $\delta\leq\xi<\delta+\omega$.
We write $\xi=\delta+\ell$ where $\ell<\omega$.
Consider $\supp( p)=\{\delta_0<\ldots<\delta_n\}$ and $ p(\delta_i)=(D_i,a_i)$ with $D_i\in\calF^\omega_{k_i}$.

\textbf{Case 1:} $\delta=\delta_{j}$ for some $j\leq n$.
Find $F\in\calF^\omega$ with canonical decomposition $\bigcup_{i<n_k}F_i$, such that $D_n\cup\{a_j+1,\ldots,a_j+\ell\}\subset F_0$ and $R(F)=F_0\cap a_j$.
For every $i\geq j$, Apply Lemma~\ref{construction.properties2} to $D_i$, $F$, and $a_i$,
to find $W_i\in\calF_{k_i}$ such that $|W_i\cap a_i|=|D_i\cap a_i|$ and $W_i\setminus a_i$ is an interval of $F_0$ 
with $a_i\in W_i$. Note that for every $i\geq j$, $a_i \in F_0\setminus R(F)$, and so $W_i^*=\phi_1(W_i)$, $a_i^*=\phi_1(a_i)$, have that $|W_i^*\cap a_i^*|=|D_i\cap a_i|$ and $W_i^*\setminus a_i^*$ is an interval of $F$.
Define $ q\in\PP$ so that $\supp( q)=\supp( p)$ and 
$$ q(\delta_i)=\begin{cases}(D_i,a_i) &\mbox{for }i<j, \\
               (F,a_i^*) &\mbox{for }i\geq j.
							\end{cases}$$
Note that the $W_i$'s witness $ q\leq p$. By construction $\Phi^q(a_i+\ell)=\xi$.

\textbf{Case 2:} $\delta\notin\supp( p)$.
If $\delta>\delta_n$ it is easy to find $ q$, we leave the reader to work out the details.
Assume then there is $j<n$ with $\delta<\delta_{j}$.
Pick $a>D_n$ and apply \eqref{omegaCS} to find $F\in\calF^\omega$ with canonical decomposition $\bigcup_{i<n_k}F_i$, 
$D_n\cup\{a,a+1,\ldots,a+\ell\}\subset F_0$ and $R(F)=F_0\cap a_{j}$.
Apply Lemma~\ref{construction.properties2} to find $W_i^*\in\calF^\omega_{k_i}$ for $i\geq j$,
such that $|W_i^*\cap a_i|=|D_i\cap a_i|$, $W_i$ is an interval of $F$ with $a_i\in W^*_i$, for $i\geq j$.
Now let $c_i=\varphi_{1}(a_i)$, and $W_i=\varphi_1(W_i^*)$ for $i\geq j$.

Define $ q\in\PP$ with $\supp( q)=\supp( p)\cup\{\delta\}$ such that
$$ q(\gamma)=\begin{cases} (D_i,a_i) &\mbox{for }\gamma=\delta_i, i< j.\\
                           (F,a)     &\mbox{for }\gamma=\delta, \\
													 (F,c_i)   &\mbox{for }\gamma=\delta_i, i\geq j.
						\end{cases}
$$
It's clear that $ q\leq p$ (this is witness by the $W_i$'s) and $\Phi^q(a+\ell)=\xi$ by construction.
\end{proof}

\begin{claim}
$\calF$ as above is a construction scheme on $V[G]$.
\end{claim}
\begin{proof}
To see property (1) of a Construction Scheme let $A\subset\omega_1$ finite and $ p\in\PP$.
We can find $ q\leq p$ and $B\subset\omega$ finite such that $\Phi^q(B)=A$.
If $\supp( q)=\{\delta_0<\ldots<\delta_n\}$ and $ q(\delta_i)=(D_i,a_i)$, 
we can find $F\in\calF^\omega$ with $B\cup D_n\subset F$.
Define $ q_0\in\PP$ so that $\supp( q_0)=\supp( q)$ and $ q_0(\delta_i)=(F,a_i)$.
Using Lemma~\ref{construction.properties2}, it is easy to check $ q_0\leq q$. It is also clear that $ q_0$ forces \textit{``there exists $F\in\dot{\calF}$ with $A\subset F$''.}
This shows property $(1)$ of a Construction Scheme.

Properties (2)--(4) are easy to check by contradiction. 
\end{proof}

We have $\calF$ on $V[G]$ a construction scheme on $V[G]$.
To show $\calF$ is capturing, let $\dot{S}=(s_\alpha:\alpha<\omega_1)$ be an uncountable $\Delta$-System on $V[G]$.
Assume for simplicity $|s_\alpha|=1$ and $\dot{S}=(\{\xi\}:\xi\in\dot{\Gamma})$ where $\dot{\Gamma}$ is a name for an uncountable 
subset of $\omega_1$, the proof is the same for the general case.
Let $n<\omega$ be given.
Take $\Omega\subset\omega_1$ uncountable and $ p_\alpha\in\PP$ for $\alpha\in\Omega$ such that
\begin{equation}\label{cohen.cap.eq2}
 p_\alpha\Vdash\alpha\in\dot{\Gamma}
\end{equation}
We can assume without loss of generality that there is $\delta\in\supp( p_\alpha)$ such that 
$ p_\alpha(\delta)=(D,a)$, and $\alpha\in \phi_{a,\delta}(D)$. 

Find $\Omega_0\subset\Omega$ uncountable, $\delta_{\alpha,0}<\ldots<\delta_{\alpha,d-1}<\omega_1$ limit,
$D_i\in\calF^\omega_{k_i}$ for $i<d$, $a_0<\ldots<a_{d-1}$, and $x<\omega$ such that:
\begin{enumerate}
\item $(\supp( p_\alpha):\alpha\in\Omega_0)$ form a $\Delta$-System with root $\{\delta_{\alpha,0},\ldots,\delta_{\alpha,r-1}\}$,
\item $\supp( p_\alpha)=\{\delta_{\alpha,0},\ldots,\delta_{\alpha,d-1}\}$,
\item $ p_\alpha(\delta_{\alpha,i})=(D_i,a_i)$ for every $i<d$,
\item For $x\in D_{d-1}$ with $\Phi^{p_\alpha}(x)=\alpha$, there is fixed $j_0$ with: $j_0=d-1$ if $x\geq a_{d-1}$, or $j_0<d-1$ and is such that $a_{j_0}\leq x<a_{j_0+1}$.
\end{enumerate}

Pick $\alpha_0<\ldots<\alpha_{n-1}$ in $\Omega_0$.
We want to find $ q\in\PP$ such that 
\begin{equation}\label{cohen.cap.eq3}
 q\Vdash \alpha_i\in\dot{\Gamma}, \dot{\calF}\mbox{ captures } \alpha_0,\ldots,\alpha_{n-1}.
\end{equation}

Apply \eqref{omegaCS} to find $F^*\in\calF^\omega_k$ such that $k>k_{d-1}$,
$F^*=\bigcup_{i<n_k}F^*_i$ is the canonical decomposition of $F^*$,
$D_{d-1}\subset F^*_0$, and $R(F^*)=F^*_0\cap a_r$.

For $i<d$, note $a_i\in D_i\subset F^*_0$, 
therefore we can apply Lemma~\ref{construction.properties2} to find $W_{0,i}\in\calF^\omega_{k_i}$
with $W_{0,i}\cap a_i =D_i\cap a_i$ and $W_{0,i}\setminus a_i$ an interval of $F^*_0$ with $a_i\in W_{0,i}$.
Let $\varphi_i:F^*_0\rightarrow F_i^*$ be the increasing bijection between $F^*_0$ and $F_i^*$.
Define $W_{i,j}=\varphi_i(W_{0,j})$, and $a_{i,j}=\varphi_{D_j,W_{i,j}}(a_j)$ for $i<n$, $j<d$,
and $x_i=\varphi_{D_{j_0},W_{i,j_0}}(x)$ for $i<n$.

It is easy to check that
\begin{equation}\label{cohen.cap.eq4}
W_{i,j}\in\calF^\omega_{k_j},|W_{i,j}\cap a_{i,j}|=|D_j\cap a_j|,
\mbox{ and }W_{i,j}\setminus a_{i,j}\mbox{ is an interval of $F^*$ with $a_{i,j}\in W_{i,j}$}
\end{equation}
and by equation~\eqref{cohen.cap.eq1} we have
\begin{equation}\label{cohen.cap.eq5}
\phi_{a_{i,j_0},\delta_{\al_i,j_0}}(x_i)=\alpha_i
\end{equation}

We define $ q\in\PP$ with $\supp( q)=\{\delta_{\alpha_i,j}:i<n,j<d\}$.
Note now that $\delta_{\alpha,i}$ does not depend on $\alpha$ for $i<r$.
$$ q(\delta_{\alpha_i,j})=\begin{cases} (D_j,a_j) &\mbox{ for }j<r\\
                                      (F^*,a_{i,j}) &\mbox{ for }j\geq r
												  \end{cases}$$

With this definition, we have $\Phi^q(x_i)=\phi_{a_{i,j_0}},\delta_{\al_i,j_0}(x_i)=\alpha_i$ by~\eqref{cohen.cap.eq5},
and  $(W_{i,j}:r\leq j<d)$ is a witness to $ q\leq p_{\alpha_i}$ for every $i<n$, by~\eqref{cohen.cap.eq4}.
This implies $ q\Vdash \alpha_i\in\dot{\Gamma}$ for every $i<n$ because of~\eqref{cohen.cap.eq2}.

Finally, let $F=\Phi^q(F^*)$. Then $ q\Vdash F\in\dot{\calF}$, and 
by the construction of $(x_i:i<n)$ and~\eqref{cohen.cap.eq5} we have $ q$ forces $F$ captures $\alpha_0,\ldots,\alpha_{n-1}$.
Therefore~\eqref{cohen.cap.eq3} holds which is what we wanted to prove.

Suppose now $\kappa>\aleph_1$. Let $\CC_\kappa$ be the forcing for adding $\kappa$ Cohen reals.
It is well known (see for example Theorem 8.2.1 of Kunen~\cite{Kunen}) that 
$\CC_\kappa=\CC_{\omega_1}\ast\CC_{\kappa\setminus\omega_1}$.
We know that $\CC_{\omega_1}$ adds capturing construction schemes, 
by Lemma~\ref{lem1}, forcing with $\CC_{\kappa\setminus\omega_1}$ preserves capturing since it has precaliber $\aleph_1$.
Therefore forcing with $\CC_\kappa$ adds capturing construction schemes.

\end{proof}

\section{Fully Capturing and Capturing with Partitions}\label{extra}

There is a generalization of capturing that proves useful in some examples of \cite{todor}.
We present it here for completeness. 
\begin{defi}
Let $\calF$ be a construction scheme. We say that $\calF$ is \emph{fully capturing} if for every uncountable $\Delta$-system
$(s_\xi)_{\xi<\omega_1}$ of finite subsets of $\omega_1$ with root $s$, and every $k^*<\omega$ there are $F\in\calF_k$ with $k>k^*$,
 canonical decomposition $F=\bigcup_{i<n_k}F_i$, and $\xi_0<\ldots<\xi_{n_k-1}<\omega_1$, such that
\begin{align*}
                                                      &s                  \subset R(F)\\
   \mbox{for every $i<n_k$,} \quad  &                    s_{\xi_i}\setminus s  \subset F_i\setminus R(F), \\
 \mbox{for every $i<n_k$,} \quad     &      \varphi_i(s_{\xi_0})=s_{\xi_i}.
\end{align*}
\end{defi}

\begin{defi}
Let $\omega=\bigcup_{\ell<\omega}P_\ell$ be a partition of $\omega$ into infinite components and let $\vec{P}=(P_\ell:\ell<\omega)$.
Suppose $(m_k,n_k,r_k)$ forms a type  such that for every $\ell<\omega$, and every $r<\omega$ there are infinitely many $k$'s in $P_\ell$
with $r_k=r$. Then we say $(m_k,n_k,r_k)_k$ forms a $\vec{P}$-\emph{type}. 
\end{defi}
\begin{defi}
Let $\calF$ be a construction scheme with type $(m_k,n_k,r_k)_k$, and $2\geq n$.
We say $\calF$ is $n$-$\vec{P}$-\emph{capturing} if $(m_k,n_k,r_k)_k$ forms a $\vec{P}$-type, and for every uncountable $\Delta$-system
$(s_\xi)_{\xi<\omega_1}$ of finite subsets of $\omega_1$ with root $s$, and every $\ell<\omega$,
there are $\xi_0<\ldots<\xi_{n-1}<\omega_1$, $k\in P_\ell$ and
$F\in\calF_k$ with canonical decomposition $F=\bigcup_{i<n_k}F_i$, such that
\begin{align*}
&s                  \subset R(F)\\
   \mbox{for every $i<n$,} \quad              &         s_{\xi_i}\setminus s \subset F_i\setminus R(F), \\
 \mbox{for every $i<n$,} \quad         &  \varphi_i(s_{\xi_0})=s_{\xi_i}.
\end{align*}

We say $\calF$ is $\vec{P}$-\emph{capturing} if $\calF$ is $n$-$\vec{P}$-capturing for every $n<\omega$.
\end{defi}

We can also define $\vec{P}$-\emph{fully capturing} in the obvious manner.
What makes this version interesting is that it allows for different amalgamations.
For example, the existence of a 2-$\vec{P}$-capturing construction scheme implies there are Suslin trees and T-gaps.
This can be shown following \cite{LT} where the same objects are build using 3-capturing.

We prove the following Theorem about the consistency of other forms of capturing.

\begin{theo}
Adding $\kappa\geq \aleph_1$ Cohen reals implies there are Fully capturing construction schemes, $\vec{P}$-capturing consctruction schemes, and $\vec{P}$-fully capturing construction schemes. 
\end{theo}
\begin{proof}
The proof is similar to Theorem~\ref{ch5.cohen} therefore we only give a sketch for a 
$\vec{P}$-fully capturing construction scheme.

Let $\vec{P}$ be a partition of $\omega$ and let $(m_k,n_k,r_k)_{k<\omega}$ be a given $\vec{P}$-type.

It is easy to see, using the fact that $(m_k,n_k,r_k)_{k<\omega}$ is a $\vec{P}$-type, 
that there is a Construction Scheme $\calF^\omega$ on $\omega$ such that:
\begin{equation}\label{CS}
\mbox{\begin{minipage}{.9\textwidth}
For every $\ell<\omega$, $A\subset\omega$ finite, and $a<\omega$, there is $k\in P_\ell$ and $F\in\calF_k$ with canonical decomposition
$\bigcup_{i<n_k}F_i$, such that $A\subset F_0$ and $R(F)=F_0\cap a$.
\end{minipage}}
\end{equation}

Suppose now $\dot{\calF}$ is defined as in Theorem~\ref{ch5.cohen} 
and 
$\dot{S}=(s_\alpha:\alpha<\omega_1)$ is an uncountable $\Delta$-System on $V[G]$.
We assume that $|s_\alpha|=1$ and $\dot{S}=(\{\xi\}:\xi\in\dot{\Gamma})$ where $\dot{\Gamma}$ is a name for an uncountable 
subset of $\omega_1$, the argument is the same for the general case.
Let $\ell<\omega$ and $k^*<\omega$ be given.

Find $\Omega\subset\omega_1$ uncountable and $ p_\alpha\in\PP$ for $\alpha\in\Omega$ such that
\begin{equation}\label{fully.1}
 p_\alpha\Vdash\alpha\in\dot{\Gamma}
\end{equation}
And there is $\delta\in\supp( p_\alpha)$ such that 
$ p_\alpha(\delta)=(D,a)$, and $\alpha\in\phi_{a,\delta}(D)$. 
And $\delta_{\alpha,0}<\ldots<\delta_{\alpha,d-1}<\omega_1$ limit,
$D_i\in\calF^\omega_{k_i}$ for $i<d$, $a_0<\ldots<a_{d-1}$, and $x<\omega$ such that:
\begin{enumerate}
\item $(\supp( p_\alpha):\alpha\in\Omega_0)$ form a $\Delta$-System with root $\{\delta_{\alpha,0},\ldots,\delta_{\alpha,r-1}\}$,
\item $\supp( p_\alpha)=\{\delta_{\alpha,0},\ldots,\delta_{\alpha,d-1}\}$,
\item $ p_\alpha(\delta_{\alpha,i})=(D_i,a_i)$ for every $i<d$,
\item For $x\in D_{d-1}$ with $\Phi^{p_\alpha}(x)=\alpha$, there is fixed $j_0$ with: $j_0=d-1$ if $x\geq a_{d-1}$, or $j_0<d-1$ and is such that $a_{j_0}\leq x<a_{j_0+1}$.
\end{enumerate}

Apply \eqref{CS} to find $k\in P_\ell$ with $k>k^*$, and $F^*\in\calF^\omega_k$ such that $k>k_{d-1}$,
$F^*=\bigcup_{i<n_k}F^*_i$ is the canonical decomposition of $F^*$,
$D_{d-1}\subset F^*_0$, and $R(F^*)=F^*_0\cap a_r$.

Pick arbitrary $\alpha_0<\ldots<\alpha_{n_k-1}$ in $\Omega$.
We construct $ q\in\PP$, such that 
\begin{equation}\label{fully.2}
 q\Vdash \alpha_i\in\dot{\Gamma}, \exists F\in\dot{\calF}_k\mbox{ captures } \alpha_0,\ldots,\alpha_{n_k-1}.
\end{equation}

For $i<d$, note $a_i\in D_i\subset F^*_0$, 
therefore we can apply Lemma~\ref{construction.properties2} to find $W_{0,i}\in\calF^\omega_{k_i}$
with $W_{0,i}\cap a_i =D_i\cap a_i$ and $W_{0,i}\setminus a_i$ an interval of $F^*_0$ with $a_i\in W_{0,i}$.
Let $\varphi_i:F^*_0\rightarrow F_i^*$ be the increasing bijection between $F^*_0$ and $F_i^*$.
Define $W_{i,j}=\varphi_i(W_{0,j})$, and $a_{i,j}=\varphi_{D_j,W_{i,j}}(a_j)$ for $i<n$, $j<d$,
and $x_i=\varphi_{D_{j_0},W_{i,j_0}}(x)$ for $i<n_k$.

It is easy to check that
\begin{equation}\label{fully.3}
W_{i,j}\in\calF^\omega_{k_j},|W_{i,j}\cap a_{i,j}|=|D_j\cap a_j|,
\mbox{ and }W_{i,j}\setminus a_{i,j}\mbox{ is an interval of $F^*$ with $a_{i,j}\in W_{i,j}$}
\end{equation}
and as before we have
\begin{equation}\label{fully.4}
\phi_{a_{i,j_0},\delta_{\al_i,j_0}}(x_i)=\alpha_i
\end{equation}

We define $ q\in\PP$ with $\supp( q)=\{\delta_{\alpha_i,j}:i<n,j<d\}$.
Note now that $\delta_{\alpha,i}$ does not depend on $\alpha$ for $i<r$.
$$ q(\delta_{\alpha_i,j})=\begin{cases} (D_j,a_j) &\mbox{ for }j<r\\
                                      (F^*,a_{i,j}) &\mbox{ for }j\geq r
												  \end{cases}$$

With this definition, we have $\Phi^q(x_i)=\phi_{a_{i,j_0}},\delta_{\al_i,j_0}(x_i)=\alpha_i$ by~\eqref{cohen.cap.eq5},
and  $(W_{i,j}:r\leq j<d)$ is a witness to $ q\leq p_{\alpha_i}$ for every $i<n_k$, by~\eqref{fully.3}.
This implies $ q\Vdash \alpha_i\in\dot{\Gamma}$ for every $i<n_k$ because of~\eqref{fully.1}.

Finally, let $F=\Phi^q(F^*)$. Then $ q\Vdash F\in\dot{\calF}$, and 
by the construction of $(x_i:i<n_k)$ and~\eqref{fully.4} we have $ q$ forces $F$ captures $\alpha_0,\ldots,\alpha_{n_k-1}$.
Therefore~\eqref{fully.2} holds which is what we wanted to prove.

\end{proof}

We also have the following results related to the consistency of $n$-$\vec{P}$-capturing.
The proof follows the arguments in Section~\ref{ncap.nkna} and it is left to the reader.
\begin{theo}
Let $\vec{P}$ be a partition of $\omega$ as above.
Then $n$-$\vec{P}$-capturing and MA$_{\omega_1}($K$_m$) are independent if $n\leq m$ and they are incompatible if $n>m$.
Also $\vec{P}$-capturing, $\vec{P}$-fully capturing, and fully capturing are all independent of MA$_{\omega_1}($precaliber $\aleph_1$).
\end{theo}

It is clear that $n$-$\vec{P}$-capturing implies $n$-capturing and $\vec{P}$-capturing implies capturing,
however we do not know if any of the implications can be reversed.
Analogously, fully capturing implies capturing but we do not know if it is consistent to have capturing without fully capturing.


\begin{thebibliography}{10}
\parskip=3mm

\bibitem{barnett}{
Janet, H. Barnett, \emph{Weak variants of Martin's Axiom}, Fund. Math. \textbf{141} (1992) 61--73
}

\bibitem{BGT}{
M. Bell, J. Ginsburg, and S. Todor\v{c}evi\'{c}, \emph{Countable spread of exp$Y$ and $\lambda Y$}, Topology Appl. \textbf{14}, 1 (1982), 1--12.
}


\bibitem{Kunen}{Kenneth Kunen, Set Theory An Introduction To Independence Proofs. 
Studies in logic and the foundations of mathematics, North Holland, 1980.
}

\bibitem{LAT}{
J. L\'{o}pez-Abad and S. Todor\v{c}evi\'{c}, \emph{Generic Banach spaces and generic simplexes}, J. Funct. Anal. \textbf{261} (2011), 300--386.
}

\bibitem{ful}{
Fulgencio L\'{o}pez, \emph{Banach spaces from a construction scheme}, J. Math. Anal. Appl. \textbf{446} (2017), 426--435.
}

\bibitem{LT}{
Fulgencio L\'{o}pez and Stevo Todor\v{c}evi\'{c}, \emph{Trees and gaps from a construction scheme}, Proc. Am. Math. Soc. \textbf{145} 2 (2017), 871--879.
}

\bibitem{Tcomb}{
Stevo Todor\v{c}evi\'{c}, Partition Problems in Topology, \emph{Contemporary Mathematics}, Volume 84 (1989).
}


\bibitem{todor}{
Stevo Todor\v{c}evi\'{c}, \emph{A construction scheme for non-separable structures}, Adv. Math. \textbf{313} (2017) 564--589
}

\end{thebibliography}
\end{document}